\documentclass[11pt]{amsart}
\usepackage{amssymb}
\usepackage{amsmath}
\usepackage{amsfonts}
\usepackage{graphicx}

\usepackage{hyperref}
    \usepackage{aeguill}
    \usepackage{type1cm}

\theoremstyle{plain}

\newtheorem{claim}{Claim}

\newtheorem{corollary}{Corollary}

\newtheorem{definition}{Definition}
\newtheorem{example}{Example}

\newtheorem{lemma}{Lemma}

\newtheorem{proposition}{Proposition}
\newtheorem{remark}{Remark}

\newtheorem{theorem}{Theorem}
\newcommand{\N}{\mathbb{N}}

\newcommand{\R}{\mathbb{R}}

\numberwithin{equation}{section}

\begin{document}

\title[Global approximation of convex functions]{Global approximation of convex functions}
\author{Daniel Azagra}
\address{ICMAT (CSIC-UAM-UC3-UCM), Departamento de An{\'a}lisis Matem{\'a}tico,
Facultad Ciencias Matem{\'a}ticas, Universidad Complutense, 28040, Madrid, Spain}
\email{daniel\_azagra@mat.ucm.es}

\date{December 2, 2011}


\dedicatory{Dedicated to the memory of Robb Fry}


\begin{abstract}
Let $U\subseteq\R^n$ be open and convex. We show that every
(not necessarily Lipschitz or strongly) convex function $f:U\to\R$ can be approximated by real
analytic convex functions, uniformly on all of $U$. In doing so we
provide a technique which transfers results on uniform
approximation on bounded sets to results on uniform approximation
on unbounded sets, in such a way that not only convexity and $C^k$
smoothness, but also local Lipschitz constants, minimizers, order,
and strict or strong convexity, are preserved. This transfer
method is quite general and it can also be used to obtain new
results on approximation of convex functions defined on Riemannian
manifolds or Banach spaces. We also provide a characterization of
the class of convex functions which can be uniformly approximated
on $\R^n$ by strongly convex functions. Finally, we give some
counterexamples showing that $C^0$-fine approximation of convex
functions by smooth convex functions is not possible on $\mathbb{R}^{n}$
whenever $n\geq 2$.
\end{abstract}

\maketitle

\section{Introduction and main results}

Two important classes of functions in analysis are those of
Lipschitz functions and convex functions $f:U\subseteq \R^n\to\R$.
Although these functions are almost everywhere differentiable (or
even a.e. twice differentiable in the convex case), it is sometimes useful
to approximate them by smooth functions which are
Lipschitz or convex as well.

In the case of a Lipschitz function $f:U\subseteq\R^n\to\R$, this can easily
be done as follows: by considering the function
$x\mapsto \inf_{y\in U}\{f(y)+L|x-y|\}$ (where
$L=\textrm{Lip}(f)$, the Lipschitz constant of $f$), which is a
Lipschitz extension of $f$ to all of $\R^n$ having the same
Lipschitz constant, one can assume $U=\R^n$. Then, by setting
$f_{\varepsilon}=f*H_{\varepsilon}$, where
$H_{\varepsilon}(x)=\frac{1}{(4\pi
    \varepsilon)^{n/2}}\exp(-|x|^2/4\varepsilon)$ is the heat kernel,
    one obtains real analytic
Lipschitz functions (with the same Lipschitz constants as $f$)
which converge to $f$ uniformly on all of $\R^n$ as
$\varepsilon\searrow 0$. If one replaces $H_\varepsilon$ with any
approximate identity $\{\delta_\varepsilon\}_{\varepsilon>0}$ of
class $C^k$, one obtains $C^k$ Lipschitz approximations. Moreover,
if $\delta_{\varepsilon}\geq 0$ and $f$ is convex, then
the functions $f_\varepsilon$ are convex as well.

However, if $f:\R^n\to\R$ is convex but not globally Lipschitz, the
convolutions $f*H_{\varepsilon}$ may not be well defined or, even
when they are well defined, they do not converge to $f$ uniformly
on $\R^n$. On the other hand, the convolutions
$f*\delta_\varepsilon$ (where
$\delta_{\varepsilon}=\varepsilon^{-n}\delta(x/\varepsilon)$,
$\delta\geq 0$ being a $C^\infty$ function with bounded support and
$\int_{\R^n}\delta=1$) are always well defined, but they only
provide uniform approximation of $f$ on {\em bounded} sets. Now,
partitions of unity cannot be used to glue
these local convex approximations into a global approximation,
because they do not preserve convexity. To see why this is so, let
us consider the simple case of a $C^2$ convex function
$f:\R\to\R$, to be approximated by $C^\infty$ convex functions.
Take two bounded intervals $I_1\subset I_2$, and $C^\infty$
functions $\theta_1, \theta_2:\R\to [0,1]$ such that
$\theta_1+\theta_2=1$ on $\R$, $\theta_1 =1$ on $I_1$, and
$\theta_2=1$ on $\R\setminus I_2$. Given $\varepsilon_j>0$ one may
find $C^\infty$ convex functions $g_j$ such that $\max\{|f-g_j|,
|f'-g'_j|, |f''-g''_j|\}\leq\varepsilon_j$ on $I_j$. If
$g=\theta_1 g_1+\theta_2 g_2$ one has
    $$
g''=g''_1\theta_1+g''_2\theta_2+ 2(g'_1-g'_2)\theta'_1
+(g_1-g_2)\theta''_1.
    $$
If $f''>0$ on $I_2$ then by choosing $\varepsilon_i$ small enough
one can control this sum and get $g''\geq 0$, but if the $g''_i=0$
vanish somewhere there is no way to do this (even if we managed to
have $g_2\geq g_1$ and $g'_2\geq g'_1$, as $\theta''_1$ must
change signs).

In \cite{Greene5, Greene3, Greene4} Greene and Wu studied the
question of approximating a convex function defined on a
(finite-dimensional) Riemannian manifold $M$\footnote{In
Riemannian geometry convex functions have been used, for instance,
in the investigation of the structure of noncompact manifolds of
positive curvature by Cheeger, Greene, Gromoll, Meyer, Siohama, Wu
and others, see \cite{GromollMeyer, CheegerGromoll, Greene1,
Greene2, Greene3, Greene4}. The existence of global convex
functions on a Riemannian manifold has strong geometrical and
topological implications. For instance \cite{Greene1}, every
two-dimensional manifold which admits a global convex function
that is locally nonconstant must be diffeomorphic to the plane,
the cylinder, or the open M{\"o}bius strip.}, and they showed that if
$f:M\to\R$ is strongly convex (in the sense of the following
definition), then for any number $\varepsilon>0$ one can find a
$C^\infty$ strongly convex function $g$ such that
$|f-g|\leq\varepsilon$ on all of $M$.

\begin{definition}
{\em A $C^{2}$ function $\varphi:M\to\mathbb{R}$
is called strongly convex if its second derivative along any nonconstant
geodesic is strictly positive everywhere on the geodesic. A (not
necessarily smooth) function $f:M\to\mathbb{R}$ is said to be
strongly convex provided that for every $p\in M$ and every
$C^{\infty}$ strongly convex function $\varphi$ defined on a
neighborhood of $p$ there is some $\varepsilon>0$ such that
$f-\varepsilon\varphi$ is convex on the neighborhood.\footnote{We
warn the reader that, in Greene and Wu's papers, what we have just
defined as strong convexity is called strict convexity. We have
changed their terminology since we will be mainly concerned with
the case $M=\R^n$, where one traditionally defines a strictly
convex function as a function $f$ satisfying $f\left(
(1-t)x+ty\right) < (1-t)f(x)+tf(y)$ if $0<t<1$.}}
\end{definition}

This solves the problem when the given function $f$ is strongly
convex. However, as Greene and Wu pointed out, their method cannot
be used when $f$ is not strongly convex. This is inconvenient
because strong convexity is a very strong condition: for instance,
the function $f(x)=x^4$ is strictly convex, but not strongly
convex on any neighborhood of $0$. However, as shown by Smith in
\cite{Smith}, this is a necessary condition in the general
Riemannian setting: for each $k=0,1,..., \infty$, there exists a
flat Riemannian manifold $M$ such that on $M$ there is a $C^k$
convex function which cannot be globally approximated by a
$C^{k+1}$ convex function (here $C^{\infty+1}$ means real
analytic). There are no results characterizing the manifolds on
which global approximation of convex functions by smooth convex
functions is possible. Even in the most basic case $M=\R^n$,
we have been unable to find any reference dealing with the problem of finding smooth global
approximations of (not necessarily Lipschitz or strongly) convex functions.

The main purpose of this note is to prove the following.

\begin{theorem}\label{uniform approximation of convex by real analytic convex}
Let $U\subseteq\R^n$ be open and convex. For every convex
function $f: U\to\R$ and every $\varepsilon>0$ there exists a
real-analytic convex function $g:U\to\R$ such that $f-\varepsilon
\leq g\leq f$.
\end{theorem}

This result is optimal in several ways, as in general it is not possible to obtain $C^{0}$-fine approximation of convex functions by $C^1$ convex functions on $\R^n$ when $n\geq 2$ (and even in the case $n=1$ this kind of approximation is not possible from below); see the counterexamples in Section 7 below.

In showing this theorem we will develop a gluing technique for
convex functions which will prove to be useful also in the setting
of Riemannian manifolds or Banach spaces.

\begin{definition}\label{defn approx from below}
{\em Let $X$ be $\R^n$, or a Riemannian manifold (not necessarily
finite-dimensional), or a Banach space, and let $U\subseteq X$ be open and convex.
We will say that a continuous convex function $f:U\to\R$ can
be approximated from below by $C^k$ convex functions, uniformly on
bounded subsets of $U$, provided that for every bounded set
$B$ with $\textrm{dist}(B, \partial U)>0$ and every $\varepsilon>0$ there
exists a $C^k$ convex function $g:U\to\R$ such that
\begin{enumerate}
\item $g\leq f$ on $U$, and
\item $f-\varepsilon\leq g$ on $B$.
\end{enumerate}
}
\end{definition}
(In the case $U=X$ we will use the convention that $\textrm{dist}(B, \partial U)=\infty$ for every bounded set $B\subset X$.)

\begin{theorem}[Gluing convex approximations]\label{main theorem}
Let $X$ be $\R^n$, or a Riemannian manifold (not necessarily
finite-dimensional), or a Banach space, and let $U\subseteq X$ be open and convex.
Assume that $U=\bigcup_{n=1}^{\infty}B_n$, where the $B_n$ are open bounded
convex sets such that $\textrm{dist}(B_n, \partial U)>0$ and $\overline{B_n}\subset B_{n+1}$
for each $n$. Assume also that $U$ has the property that every continuous, convex function $f:U\to\R$ can be
approximated from below by $C^k$ convex (resp. strongly convex) functions ($k\in\N\cup\{\infty\}$),
uniformly on bounded subsets of $U$.

Then every continuous convex function $f:U\to\R$ can be
approximated from below by $C^k$ convex (resp. strongly convex) functions, uniformly on
$U$.
\end{theorem}

From this result (and from its proof and the known results on approximation on bounded sets) we will easily deduce the
following corollaries.

\begin{corollary}\label{uniform approximation of convex functions by smooth convex functions in Rn}
Let $U\subseteq\R^n$ be open and convex. For every convex function
$f: U\to\R$ and every $\varepsilon>0$ there exists a $C^\infty$
convex function $g:U\to\R$ such that $f-\varepsilon \leq g\leq f$.
Moreover $g$ can be taken so as to preserve local Lipschitz
constants of $f$ (meaning $\textrm{Lip}(g_{|_B})\leq
\textrm{Lip}(f_{|_{(1+\varepsilon)B}})$ for every ball $B\subset
U$). And if $f$ is strictly (or strongly) convex, so can $g$ be
chosen.
\end{corollary}

\begin{corollary}\label{uniform approximation of convex functions on Cartan Hadamard manifolds}
Let $M$ be a Cartan-Hadamard Riemannian
manifold (not necessarily finite dimensional), and $U\subseteq M$ be open and convex. For every
convex function $f: U\to\R$ which is bounded on bounded subsets $B$ of $U$ with $\textrm{dist}(B, \partial U)>0$,
and for every $\varepsilon>0$ there exists
a $C^1$ convex function $g:U\to\R$ such that $f-\varepsilon \leq
g\leq f$. Moreover $g$ can be chosen so as to preserve the set of
minimizers and the local Lipschitz constants of $f$. And, if $f$ is strictly convex, so can $g$ be
taken.
\end{corollary}

One should expect that the above corollary is not optimal
(in that approximation by $C^{\infty}$ convex functions should be possible).

\begin{corollary}\label{uniform approximation of convex functions on Banach spaces}
Let $X$ be a Banach space whose dual is locally uniformly convex, and $U\subseteq X$ be open and convex.
For every convex function $f: U\to\R$ which is bounded on bounded subsets $B$ of $U$ with
$\textrm{dist}(B, \partial U)>0$, and for every
$\varepsilon>0$ there exists a $C^1$ convex function $g:U\to\R$
such that $f-\varepsilon \leq g\leq f$. Moreover $g$ can be taken
so as to preserve the set of minimizers and the local Lipschitz
constants of $f$. And if $f$ is strictly convex, so
can $g$ be taken.
\end{corollary}

A question remains open whether every convex
function $f$ defined on a separable infinite-dimensional Hilbert
space $X$ which is bounded on bounded sets can be globally approximated by $C^2$ convex functions
(notice that Theorem \ref{main theorem} cannot be combined with
the results of \cite{DFH1, DFH2} on smooth and real analytic
approximation of bounded convex bodies in order to give a solution
to this problem. Although one can use these results, together with
the implicit function theorem, to find smooth convex approximations of
$f$ on a bounded set, the approximating functions obtained by this
process are not defined on all of $X$ and are not strongly convex,
hence it is not clear how to extend them to a smooth convex function below $f$ on $X$,
or even if this should be possible at all).


Finally, as a byproduct of the proof of Theorem \ref{uniform
approximation of convex by real analytic convex} we will also
obtain the following characterization of the class of convex
functions that can be globally approximated by strongly convex
functions on $\R^n$.

\begin{proposition}\label{characterization of functions that cannot be approximated by strongly convex functions}
Let $f:\R^n\to\R$ be a convex function. The following
conditions are equivalent:
\begin{enumerate}
\item $f$ cannot be uniformly approximated by strictly convex
functions.
\item $f$ cannot be uniformly approximated by strongly convex
functions.
\item There exist $k<n$, a linear projection $P:\R^n\to\R^k$, a
convex function $c:\R^k\to\R$ and a linear function $\ell:\R^n\to\R$
such that $f=c\circ P +\ell$.
\end{enumerate}
\end{proposition}

\section{The gluing technique}

In order to prove Theorem \ref{main theorem} we will use the following.

\begin{lemma}[Smooth maxima]\label{smooth maxima}
For every $\varepsilon>0$ there exists a $C^\infty$ function $M_{\varepsilon}:\R^{2}\to\R$ with the following properties:
\begin{enumerate}
\item $M_{\varepsilon}$ is convex;
\item $\max\{x, y\}\leq M_{\varepsilon}(x, y)\leq \max\{x,y\}+\frac{\varepsilon}{2}$ for all $(x,y)\in\R^2$.
\item $M_{\varepsilon}(x,y)=\max\{x,y\}$ whenever $|x-y|\geq\varepsilon$.
\item $M_{\varepsilon}(x,y)=M_{\varepsilon}(y,x)$.
\item $\textrm{Lip}(M_{\varepsilon})=1$ with respect to the norm $\|\cdot\|_{\infty}$ in $\R^2$.
\item $y-\varepsilon\leq x<x'\implies M_{\varepsilon}(x,y)<M_{\varepsilon}(x',y)$.
\item $x-\varepsilon\leq y<y'\implies M_{\varepsilon}(x,y)<M_{\varepsilon}(x,y')$.
\item $x\leq x', y\leq y' \implies M_{\varepsilon}(x,y)\leq M_{\varepsilon}(x', y')$, with a strict inequality in the case when both $x<x'$ and $y<y'$.
\end{enumerate}
\end{lemma}
We will call $M_{\varepsilon}$ a smooth maximum.
\begin{proof}
It is easy to construct a $C^{\infty}$ function $\theta:\R\to (0,
\infty)$ such that:
\begin{enumerate}
\item $\theta(t)=|t|$ if and only if $|t|\geq\varepsilon$;
\item $\theta$ is convex and symmetric;
\item $\textrm{Lip}(\theta)=1$.
\end{enumerate}
Then it is also easy to check that the function $M_{\varepsilon}$ defined by
$$
M_{\varepsilon}(x,y)=\frac{x+y+\theta(x-y)}{2}
$$
satisfies the required properties. For instance, let us check
properties $(5), (6), (7)$ and $(8)$, which are perhaps less
obvious than the others. Since $\theta$ is $1$-Lipschitz we have
\begin{eqnarray*}
    & & M_{\varepsilon}(x,y)-M_{\varepsilon}(x', y')=\frac{x-x' +y-y'+\theta(x-y)-\theta(x'-y')}{2}\leq \\
    & &\frac{(x-x)' +(y-y')+|x-x'-y+y'|}{2}=\\
    & &\max\{x-x', y-y'\}\leq \max\{|x-x'|, |y-y'|\},
\end{eqnarray*}
which establishes $(5)$. To verify $(6)$ and $(7)$, note that our function $\theta$ must satisfy $|\theta'(t)|<1 \iff  |t|<\varepsilon$. Then we have
$$
\frac{\partial M_{\varepsilon}}{\partial x }(x,y)=\frac{1}{2}\left( 1+\theta'(x-y)\right)\geq \frac{1}{2}\left( 1-|\theta'(x-y)|\right)>0 \textrm{ whenever } |x-y|<\varepsilon,
$$
while
$$
\frac{\partial M_{\varepsilon}}{\partial x }(x,y)=\frac{1}{2}\left( 1+\theta'(x-y)\right)= \left\{
                                                    \begin{array}{ll}
                                                      1, & \textrm{ if } x\geq y+\varepsilon, \\
                                                      0, & \textrm{ if } y\geq x+\varepsilon.
                                                    \end{array}
                                                  \right.
$$
This implies $(6)$ and, together with $(4)$, also $(7)$ and the first part of $(8)$. Finally, if for instance we have $x'>x=\max\{x,y\}$ then $M_{\varepsilon}(x,y)<M_{\varepsilon}(x',y)$ by $(6)$, and if in addition $y'>y$ then $M_{\varepsilon}(x',y)\leq M_{\varepsilon}(x',y')$ by the first part of $(8)$, hence $M_{\varepsilon}(x,y)<M_{\varepsilon}(x',y')$. This shows the second part of $(8)$.
\end{proof}

The smooth maxima $M_{\varepsilon}$ are useful to approximate the
maximum of two functions without losing convexity or other key
properties of the functions, as we next see.

\begin{proposition}\label{properties of M(f,g)}
Let $U\subseteq X$ be as in the statement of Theorem \ref{main theorem},
$M_{\varepsilon}$ as in the preceding Lemma, and let $f, g: U\to\R$
be convex functions. For every $\varepsilon>0$, the function
$M_{\varepsilon}(f,g):U\to\R$ has the following properties:
\begin{enumerate}
\item $M_{\varepsilon}(f,g)$ is convex.
\item If $f$ is $C^k$ on $\{x: f(x)\geq g(x)-\varepsilon\}$ and $g$ is $C^k$ on $\{x: g(x)\geq f(x)-\varepsilon\}$ then $M_{\varepsilon}(f,g)$ is $C^k$ on $U$. In particular, if $f, g$ are $C^k$, then so is $M_{\varepsilon}(f,g)$.
\item $M_{\varepsilon}(f,g)=f$ if $f\geq g+\varepsilon$.
\item $M_{\varepsilon}(f,g)=g$ if $g\geq f+\varepsilon$.
\item $\max\{f,g\}\leq M_{\varepsilon}(f,g)\leq \max\{f,g\} + \varepsilon/2$.
\item $M_{\varepsilon}(f,g)=M_{\varepsilon}(g, f)$.
\item $\textrm{Lip}(M_{\varepsilon}(f,g)_{|_B})\leq \max\{ \textrm{Lip}(f_{|_B}), \textrm{Lip}(g_{|_B}) \}$ for every ball $B\subset U$ (in particular $M_{\varepsilon}(f,g)$ preserves common local Lipschitz constants of $f$ and $g$).
\item If $f, g$ are strictly convex on a set $B\subseteq U$, then so is $M_{\varepsilon}(f,g)$.
\item If $f, g\in C^2(X)$ are strongly convex on a set $B\subseteq U$, then so is $M_{\varepsilon}(f,g)$.
\item If $f_1\leq f_2$ and $g_1\leq g_2$ then $M_{\varepsilon}(f_1, g_1)\leq M_{\varepsilon}(f_2, g_2)$.
\end{enumerate}
\end{proposition}
\begin{proof}
Properties $(2), (3), (4), (5), (6), (7)$ and $(10)$ are obvious
from the preceding lemma. To check $(1)$ and $(8)$, we simply use
$(10)$ and convexity of $f,g$ and $M_{\varepsilon}$ to see that,
for $x, y\in U$, $t\in [0,1]$,
\begin{eqnarray*}
& &
M_{\varepsilon}\left(f(tx+(1-t)y), g(tx+(1-t)y)\right)\leq \\
& &M_{\varepsilon}\left(tf(x)+(1-t)f(y), tg(x)+(1-t)g(y)\right)=\\
& &
M_{\varepsilon}\left(t(f(x), g(x))+(1-t)(f(y), g(y))\right)\leq \\
& & t M_{\varepsilon}(f(x), g(x))+(1-t)M_{\varepsilon}(f(y), g(y)),
\end{eqnarray*}
and, according to $(8)$ in the preceding lemma, the first
inequality is strict whenever $f$, $g$ are strictly convex and
$0<t<1$. To check $(9)$, it is sufficient to see that the function
$t\mapsto M_{\varepsilon}(f,g)(\gamma(t))$ has a strictly positive
second derivative at each $t$, where $\gamma(t)=x+tv$ with $v\neq 0$ (or, in the
Riemannian case, $\gamma$ is a nonconstant geodesic). So, by replacing $f, g$
with $f(\gamma(t))$ and $g(\gamma(t))$ we can assume that $f$ and
$g$ are defined on an interval $I\subseteq\R$ on which we have
$f''(t)>0, g''(t)>0$. But in this case we easily compute
\begin{eqnarray*}
& &\frac{d^2}{dt^2}M_{\varepsilon}(f(t), g(t))=\\
& &\frac{ \left(1+\theta'(f(t)-g(t))\right) f''(t) +
\left(1-\theta'(f(t)-g(t))\right) g''(t)}{2} +\\
& & + \frac{\theta''(f(t)-g(t)) \left( f(t)-g(t)\right)^{2}}{2}\geq\\
& &\geq \frac{1}{2}\min\{f''(t), g''(t)\}>0,
\end{eqnarray*}
because $|\theta'|\leq 1$ and $\theta''\geq 0$.
\end{proof}

\medskip

\begin{center}
{\bf Proof of Theorem \ref{main theorem}.}
\end{center}

Given a continuous convex function $f:U\to\R$ and $\varepsilon>0$,
we start defining $f_1=f$ and use the assumption that
$f_1-\varepsilon/2$ can be approximated from below by $C^k$ convex
functions, to find a $C^k$ convex function $h_1:U\to\R$ such that
$$ f_1 -\varepsilon\leq h_1 \textrm{
on } B_1, \textrm{ and }   h_1 \leq f_1-\frac{\varepsilon}{2} \textrm{
on } U.$$ We put $g_1=h_1$. Now define $f_2=f_1-\varepsilon$ and
find a convex function $h_2\in C^k(U)$ such that
$$f_2-\frac{\varepsilon}{2}\leq h_2 \textrm{ on } B_2, \textrm{ and } h_2 \leq f_2-\frac{\varepsilon}{4}
\textrm{ on } U.$$ Set
$$g_2=M_{\frac{\varepsilon}{10^2}}(g_1, h_2).$$ By the preceding
proposition we know that $g_2$ is a convex $C^k$ function
satisfying $$\max\{g_1, h_2\}\leq g_2\leq \max\{g_1,
h_2\}+\frac{\varepsilon}{10^{2}} \textrm{ on } U,$$ and
$$
g_{2}(x)=\max\{g_1(x), h_{2}(x)\} \textrm{ whenever } |h_{1}(x)-h_2(x)|\geq \frac{\varepsilon}{10^2}.
$$
\begin{claim}
We have
$$
g_2=g_1 \, \textrm{ on } \, B_1, \,\,\, \textrm{ and } \,\,\,
f-\varepsilon-\frac{\varepsilon}{2}\leq g_2\leq f-\frac{\varepsilon}{2}+\frac{\varepsilon}{10^2} \, \textrm{ on } \, B_2.
$$
\end{claim}
\noindent Indeed, if $x\in B_1$, $$g_1(x)\geq f_1(x)-\varepsilon=f_{2}(x)-\frac{\varepsilon}{4}+\frac{\varepsilon}{4}\geq h_2(x)+\frac{\varepsilon}{4}\geq h_2(x)+\frac{\varepsilon}{10^2},
$$
hence $g_2(x)=g_1(x)$, and in particular $f(x)-
\frac{\varepsilon}{2}\geq g_{2}(x)\geq f(x)-\varepsilon$. While,
if $x\in B_2\setminus B_1$ then
\begin{eqnarray*}
& &
f(x)-\varepsilon-\frac{\varepsilon}{2}\leq \max\{g_1(x), h_2(x)\}\leq g_{2}(x)\leq
\max\{g_1(x), h_2(x)\}+\frac{\varepsilon}{10^{2}}\leq\\
& & \max\{f(x)-\frac{\varepsilon}{2},
f(x)-\varepsilon-\frac{\varepsilon}{4}\}+
\frac{\varepsilon}{10^2}=f(x)-\frac{\varepsilon}{2}+\frac{\varepsilon}{10^2}.
\end{eqnarray*}
This proves the claim.

Next, define $f_3=f_2-\varepsilon/2=f-\varepsilon-\varepsilon/2$,
find a convex $C^k$ function $h_3$ on $U$ so that
$$
f_3-\frac{\varepsilon}{2^2}\leq h_3 \textrm{ on } B_3, \textrm{ and } h_3 \leq f_{3}-\frac{\varepsilon}{2^3} \, \textrm{ on } U,
$$
and set $$g_3=M_{\frac{\varepsilon}{10^3}}(g_2, h_3).
$$
\begin{claim} We have
$$
g_3=g_2 \, \textrm{ on } \, B_2, \,\,\, \textrm{ and } \,\,\,
f-\varepsilon-\frac{\varepsilon}{2}-\frac{\varepsilon}{2^{2}}\leq g_3
\leq f-\frac{\varepsilon}{2}+\frac{\varepsilon}{10^2}+\frac{\varepsilon}{10^3} \, \textrm{ on } \, B_3.
$$
\end{claim}
\noindent This is easily checked as before.

In this fashion we can inductively define a sequence of $C^k$ convex functions $g_n$ on $U$ such that
$$
g_n=g_{n-1} \, \textrm{ on } \, B_{n-1}, \,\,\, \textrm{ and } \,\,\,
$$
$$
f-\varepsilon-\frac{\varepsilon}{2}-\frac{\varepsilon}{2^{2}}-...-\frac{\varepsilon}{2^{n-1}}\leq g_n
\leq f-\frac{\varepsilon}{2}+\frac{\varepsilon}{10^2}+\frac{\varepsilon}{10^3}+...+\frac{\varepsilon}{10^n} \, \textrm{ on } \, B_n
$$
(at each step of the inductive process we define $f_{n}=f_{n-1}-\varepsilon/2^{n-2}=f-\varepsilon-...-\varepsilon/2^{n-2}$, we find $h_{n}$ convex and $C^k$ such that $f_{n}-\varepsilon/2^{n-1}\leq h_n$ on $B_n$ and $h_n\leq f_n-\varepsilon/2^n$ on $U$, and we put $g_n=M_{\varepsilon/10^n}(g_{n-1}, h_n)$).

Having constructed a sequence $g_n$ with such properties, we
finally define $$g(x)=\lim_{n\to\infty}g_n(x).$$ Since we have
$g_{n+k}=g_{n}$ on $B_n$ for all $k\geq 1$, it is clear that
$g=g_n$ on each $B_n$, which implies that $g$ is $C^k$ and convex on $U$
(or even strongly convex when the $g_n$ are strongly convex).
Besides, for every $x\in U=\bigcup_{n=1}^{\infty}B_n$ we have
$$
f(x)-2\varepsilon=f(x)-\sum_{n=1}^{\infty}\frac{\varepsilon}{2^{n-1}}\leq g(x)\leq f(x)-\frac{\varepsilon}{2}+\sum_{n=2}^{\infty}\frac{\varepsilon}{10^n},
$$
hence $f-2\varepsilon\leq g\leq f$. \, \, \, $\Box$

\begin{remark}\label{the method preserves strong convexity etc}
{\em From the above proof and from Proposition \ref{properties of
M(f,g)} it is clear that this method of transferring convex
approximations on bounded sets to global convex approximations
preserves strict and strong convexity, local Lipschitzness,
minimizers and order, whenever the given approximations on bounded
sets have these properties.}
\end{remark}

\section{Why approximating Lipschitz convex functions is enough}

We will deduce our corollaries by combining Theorem \ref{main theorem} with the known results on approximation of convex functions on bounded sets mentioned in the introduction, and with the following.

\begin{proposition}\label{uniform approximation of Lipschitz functions is enough}
Let $X$ be $\R^n$, or a Cartan-Hadamard manifold
(not necessarily finite-dimensional), or a Banach space, and let $U\subseteq X$ be open and convex.
Assume that $U$ has the property that every Lipschitz convex function on $U$ can be approximated
by $C^k$ convex (resp. strongly convex) functions, uniformly on $U$.

Then every convex function $f:U\to\R$
which is bounded on bounded subsets $B$ of $U$ with $\textrm{dist}(B, \partial U)>0$ can be approximated from below by $C^k$ convex
(resp. strongly convex) functions, uniformly on bounded subsets of $U$.
\end{proposition}
\begin{proof}
It is well known that a convex function
$f:U\to\R$ which is bounded on bounded subsets $B$ of $U$ with $\textrm{dist}(B, \partial U)>0$ is
also Lipschitz on each such subset $B$ of $X$.
So let $B\subset U$ be bounded, open and convex with $\textrm{dist}(B, \partial U)>0$, put
$L=\textrm{Lip}(f_{|_B})$, and define
$$
g(x)=\inf\{f(y)+L\, d(x,y) :y\in U\},
$$
where $d(x,y)=\|x-y\|$ in the case when $X$ is $\R^n$ or a Banach space, and $d$ is the Riemannian distance in $X$ when $X$ is a Cartan-Hadamard manifold.
\begin{claim}
The function $g$ has the following properties:
\begin{enumerate}
\item $g$ is convex on $X$.
\item $g$ is $L$-Lipschitz on $X$.
\item $g=f$ on $B$.
\item $g\leq f$ on $U$.
\end{enumerate}
\end{claim}
These are well known facts in the vector space case, but perhaps
not so in the Riemannian setting, so let us say a few words about
the proof. Property $(4)$ is obvious. To see that the reverse
inequality holds on $B$, take $x\in B$ and a subdifferential
$\zeta\in D^{-}f(x)$ (we refer to \cite{AFLM, AF2} for the
definitions and some properties of the Fr{\'e}chet subdifferential and
inf convolution on Riemannian manifolds). We have $\|\zeta\|_x\leq
L$ because $f$ is $L$-Lipschitz on $B$. Since $\exp_{x}:TX_{x}\to
X$ is a diffeomorphism, for every $y\in X$ there exists $v_y\in
TX_x$ such that $\exp_{x}(v_y)=y$. And, because $t\mapsto f(\exp_x(tv_y))$ is
convex, we have $f(\exp_{x}(tv_y))-f(x)\geq \langle
\zeta, tv_y\rangle_x$ for every $t$, and in particular, taking
$t=1$, we get $f(y)-f(x)\geq \langle \zeta, tv_y\rangle_x\geq
-\|\zeta\|_x \|v_y\|_x\geq -L d(x,y)$. Hence $f(y)+ L d(x,y)\geq
f(x)$ for all $y\in X$, and taking the inf we get $g(x)\geq f(x)$.
Therefore $g=f$ on $B$. Showing $(2)$ is easy (as a matter of fact
this is true in every metric space). Finally, to see that $g$ is
convex on $X$, one does have to use that $X$ is a Cartan-Hadamard
manifold. We
note that in a Cartan-Hadamard manifold $X$ the distance function
$d:X\times X\to [0, \infty)$ is globally convex (see for instance \cite[V.4.3]{Sakai} and
\cite[Corollary 4.2]{AF2}), and that if $X\times U\ni(x,y)\mapsto
F(x,y)$ is convex then $x\mapsto \inf_{y\in U}F(x,y)$ is also
convex on $X$ (see \cite[Lemma 3.1]{AF2}). Since $(x,y)\mapsto f(y)+L
d(x,y)$ is convex on $X\times U$, this shows $(1)$.

\medskip

Now, for a given $\varepsilon>0$, by assumption there exists a $C^k$ convex
(resp. strongly convex) function $\varphi:U\to\R$ so that $g-\varepsilon\leq\varphi\leq g$ on $U$.
Since $g\leq f$ on $U$, and $g=f$ on $B$, this implies that $\varphi\leq f$ on $U$,
and $f-\varepsilon\leq\varphi$ on $B$.
\end{proof}

\bigskip

\section{Uniform approximation of convex functions on $\R^n$}

\medskip

Let $f:\R^n\to\R$ be continuous. As we recalled in the
introduction, if $\delta:\R^n\to [0, \infty)$ is a $C^\infty$
function such that $\delta(x)=0$ whenever $\|x\|\geq 1$, and
$\int_{\R^n}\delta=1$, then the functions
$f_{\varepsilon}(x)=\int_{\R^n}f(x-y)\delta_{\varepsilon}(y)dy$
(where
$\delta_{\varepsilon}(x)=\varepsilon^{-n}\delta(x/\varepsilon)$)
are $C^\infty$ and converge to $f(x)$ uniformly on every compact
set, as $\varepsilon\searrow 0$. Moreover, as is well known and
easily checked:
\begin{enumerate}
\item If $f$ is uniformly continuous then $f_{\varepsilon}$ converges to $f$
uniformly on $\R^n$.
\item If $f$ is convex (resp. strictly, or strongly convex), so is $f_{\varepsilon}$.
\item If $f$ is Lipschitz, so is $f_{\varepsilon}$, and $\textrm{Lip}(f_{\varepsilon})=\textrm{Lip}(f)$.
\item If $f$ is locally Lipschitz, $\textrm{Lip}(f_{{\varepsilon}_{|_B}})=\textrm{Lip}(f_{|_{(1+\varepsilon)B}})$ for every ball $B$.
\item If $f\leq g$ then $f_{\varepsilon}\leq g_{\varepsilon}$.
\end{enumerate}
Therefore this method provides uniform approximation of Lipschitz
convex functions by $C^{\infty}$ convex functions, uniformly on
$\R^n$. By Proposition \ref{uniform approximation of Lipschitz
functions is enough} we then have that every (not necessarily
Lipschitz) convex function $f:\R^n\to\R$ can be approximated from
below by $C^\infty$ convex functions, uniformly on bounded sets.
And by Theorem \ref{main theorem} we get that every convex
function $f:\R^n\to\R$ can be approximated from below by
$C^\infty$ convex functions, uniformly on $\R^n$. Moreover, it is
clear that strict (or strong) convexity, local Lipschitzness, and
order are preserved by the combination of these techniques.

The case when $X=U$ is an open convex subset of $\R^n$ can be treated in
a similar way. We consider the open, bounded convex sets
$B_m=\{x\in U
: \textrm{dist}(x, \partial U)>1/m, \|x\| <m\}$, so we have
$\overline{B_m}\subset B_{m+1}$, $\textrm{dist}(B_m, \partial U)>0$ and $U=\bigcup_{m=1}^{\infty}B_m$.
By combining Theorem \ref{main theorem} and Proposition
\ref{uniform approximation of Lipschitz functions is enough}, it
suffices to show that every Lipschitz, convex function $f:U\to\R$
can be approximated by $C^{\infty}$ convex functions, uniformly on
$U$. This can be done as follows: set
$L=\textrm{Lip}(f)$ and consider
    $$
g(x)=\inf\{f(y)+L \|x-y\| \, : \, y\in U\}, \,\,\, x\in\R^n,
    $$
which is a Lipschitz, convex extension of $f$ to all of $\R^n$,
with $\textrm{Lip}(f)=\textrm{Lip}(g)$. By using the above
argument, $g$ can be approximated by $C^\infty$ convex functions,
uniformly on $\R^n$. In particular, $f=g_{|_U}$ can be
approximated by such functions, uniformly on $U$. This proves
Corollary \ref{uniform approximation of convex functions by smooth
convex functions in Rn}.

\section{The Riemannian and the Banach cases}

Let us see how one can deduce Corollaries \ref{uniform
approximation of convex functions on Cartan Hadamard manifolds}
and \ref{uniform approximation of convex functions on Banach
spaces}. As in the case of $\R^n$, the combination of Theorem
\ref{main theorem}, Proposition \ref{uniform approximation of
Lipschitz functions is enough} and Remark \ref{the method
preserves strong convexity etc} reduces the problem to showing
that every {\em Lipschitz} convex function $f:X\to\R$ (where $X$
stands for a Cartan-Hadamard manifold or a Banach space whose dual
is locally uniformly convex) can be approximated by $C^1$ convex
functions, uniformly on $X$. It is well known that this can be
done via the inf convolution of $f$ with squared distances: the
functions
    $$
f_{\lambda}(x)=\inf\{ f(y)+ \frac{1}{2\lambda} d(x,y)^2 \, : \,
y\in X\}
    $$
are $C^1$, convex, Lipschitz (with the same constant as $f$), have
the same minimizers as $f$, are strictly convex whenever $f$ is,
and converge to $f$ as $\lambda\searrow 0$, uniformly on all of
$X$. See \cite{Stromberg} for a survey on the inf convolution
operation in Banach spaces, and \cite{AF2} for the Cartan-Hadamard
case.

\section{Real analytic convex approximations}

Let us finally consider the question about global approximation of convex
functions $f$ by real analytic convex functions on $\R^n$. As mentioned in
the introduction, real analytic approximations of partitions of
unity cannot be employed to glue local approximations into a
uniform approximation of $f$ on all of $\R^n$, unless those local
approximations are strongly convex. Obviously, Theorem \ref{main
theorem} cannot be directly used to this purpose either, because
it only provides $C^{\infty}$ smoothness.

A natural approach to this problem could be trying to approximate
(not necessarily strongly) convex functions by real analytic
strongly convex functions on bounded sets, and then gluing all the
approximating functions by using a real analytic approximation to
a partition of unity. This would lead to cumbersome technical
problems, as the involved functions would no longer have bounded
supports, so one would have to control an infinite sum of products
of functions and their first and second derivatives, and make the
second derivatives of the convex functions prevail. An essentially
equivalent, but more efficient approach would be the following:
first, showing that every convex function can be approximated by
$C^2$ strongly convex functions, and then using Whitney's theorem
on $C^2$-fine approximation of functions by real analytic
functions to conclude. This is what we will try to do.

However, not every convex function
$f:\R^n\to\R$ can be approximated by strongly convex functions
uniformly on $\R^n$. For instance, it is not possible to
approximate a linear function by strongly convex functions. This
impossibility is not related to the fact that a linear function is
not strongly convex, but rather to the fact that the range of its
subdifferential is a singleton: if one considers $f:\R\to\R$
defined by $f(x)=\max\{ax+b, cx+d\}$, with $a\neq c$ (which is
piecewise linear, and in particular not even strictly convex),
then one immediately realizes that $f$ can be uniformly
approximated by $C^\infty$ strongly convex functions. Since every
convex function on $\R$ can be approximated by piecewise
linear convex functions, a combination of this observation with
the gluing technique of Section $2$, allows us to solve the
problem in the way we set out to do, provided that the given
function is not affine. And of course, if it is affine, it already
is real analytic, so there is nothing to show.

It $\R^2$ the situation is more complicated: for instance, the
function $f(x,y)=|x-y|$ is not affine and it cannot be
approximated by strongly convex functions. However, up to a linear
change of coordinates, $f(u,v)=|u|$, so if we construct a real
analytic convex approximation $\theta$ of $t\mapsto |t|$ on $\R$
and we define $g(u,v)=\theta(u)$ then we get a real analytic
convex approximation of $f$. On the other hand, every function of
the form $h(x_1,x_2)=\max\{a_1x_1+b_1, \, c_1x_1+d_1, \,
a_2x_2+b_2, \, c_2x_2+d_2\}$ (with $a_i\neq c_i$) can be uniformly
approximated by strongly convex functions (even though its graph
is a finite union of convex subsets of hyperplanes).

We next elaborate on these ideas to show that, given a convex
function $f:U\subseteq\R^n\to\R$, either we can reduce the problem of
approximating $f$ by real analytic convex functions to some $\R^k$
with $k<n$, or else its graph is supported by a maximum of
finitely many {\em $(n+1)$-dimensional corners} which besides
approximates $f$ on a given bounded set (and which in turn we will
manage to approximate by strongly convex functions).

\begin{definition}[Supporting corners]
{\em We will say that a function $C:\R^n\to\R$ is a
$k$-dimensional {\em corner function} on $\R^n$ if it is of the
form
    $$
C(x)=\max\{\, \ell_1 +b_1, \, \ell_2 +b_2, \, ..., \, \ell_k +b_k
\, \},
    $$
where the $\ell_j:\R^n\to\R$ are linear functions such that the
functions $L_{j}:\R^{n+1}\to\R$ defined by $L_{j}(x,
x_{n+1})=x_{n+1}-\ell_j(x)$, $1\leq j\leq k$, are linearly
independent, and the $b_j\in\R$. We will also say that a convex
function $f:U\subseteq\R^n\to\R$ is supported by $C$ at a point $x\in U$
provided we have $C\leq f$ on $U$ and $C(x)=f(x)$.}
\end{definition}

\begin{lemma}\label{strongly convex approximation of corners}
If $C$ is an $(n+1)$-dimensional corner function on $\R^n$ then
$C$ can be approximated by $C^{\infty}$ strongly convex functions,
uniformly on $\R^n$.
\end{lemma}
\begin{proof}
We will need to use the following variation of the smooth maximum of Lemma \ref{smooth maxima}: given $\varepsilon, r>0$,
let $\beta_{\varepsilon, r}=|\cdot|*H_{r}+\varepsilon/2$, where $H_{r}(x)=\frac{1}{(4\pi r)^{1/2}}\exp(-x^{2}/4 r)$. We have $\beta_{\varepsilon, r}''(t)=2e^{-t^2/4 r}/(4 r\pi)^{1/2}>0$, so $\beta_{\varepsilon, r}$ is strongly convex and $1$-Lipschitz,  and as $r\to 0$ we have $\beta_{\varepsilon, r}(t)\to |t|+\varepsilon/2$ uniformly on $t\in\R$, so we may find $r=r(\varepsilon)>0$ such that $|t|\leq \beta_{\varepsilon, r}(t)\leq |t|+ \varepsilon$ for all $t$. Put $\widetilde{\theta}_{\varepsilon}(t)=\beta_{\varepsilon, r(\varepsilon)}(t)$, and define $\widetilde{M}_{\varepsilon}:\R^2\to\R$ by
$$
\widetilde{M}_{\varepsilon}(x,y)=\frac{x+y+\widetilde{\theta}_{\varepsilon}(x-y)}{2}.
$$
It is clear that $\widetilde{M}_{\varepsilon}$ satisfies all the properties of Lemma \ref{smooth maxima} except for $(3)$.

Now let us prove our lemma. Up to an affine change of variables in $\R^{n+1}$, the problem is
equivalent to showing that the function
    $$
f(x)=\max\{ 0, x_1, x_2, ..., x_{n}\}
    $$
can be uniformly approximated on $\R^n$ by $C^{\infty}$ strongly
convex functions. We will show that this is possible by induction
on $n$.

For $n=1$, the function $f(x)=\max\{x,0\}$ is Lipschitz,
so the convolutions $f_{\varepsilon}=f*H_{\varepsilon}$  are $C^\infty$, Lipschitz and converge
to $f$, uniformly on $\R$, as $\varepsilon\searrow 0$. Besides, as
one can easily compute,
    $$
f''_{\varepsilon}(x)=\frac{1}{(4\pi\varepsilon)^{1/2}}
e^{-\frac{x^2}{4\varepsilon}}>0,
    $$
so the $f_{\varepsilon}$ are strongly convex.

Now, suppose that the function $f(x_{1}, ..., x_{k})=\max\{0,
x_1,...,x_k\}$ can be uniformly approximated by $C^\infty$ smooth
strongly convex functions on $\R^k$. Then, for a given
$\varepsilon>0$ we can find $C^\infty$ strongly convex functions
$g:\R^k\to\R$ and $\alpha:\R\to\R$ such that
    $$
f(x)\leq g(x)\leq f(x)+\varepsilon \,\,\, \textrm{ for all }
\,\,\, x\in\R^k, \,\,\, \textrm{ and } \,\,\,
    $$
    $$
\max\{t,0\}\leq\alpha(t)\leq\max\{t,0\}+\varepsilon \,\,\,
\textrm{for all }\,\,\, t\in\R.
    $$
Given the function $$F(x_1..., x_k, x_{k+1})=\max\{0, x_1, ...,
x_{k+1}\}=\max\{x_{k+1}, \, f(x_1, ..., x_k)\},$$ let us define
$G:\R^{k+1}\to\R$ by
    $$
G(x_1, ..., x_{k+1})=\widetilde{M}_{\varepsilon}\left( g(x_1, ..., x_k),
\alpha(x_{k+1})\right).
    $$
We have $G\in
C^{\infty}(\R^{k+1})$, and $F(x)\leq G(x)\leq F(x)+2\varepsilon$
for all $x\in \R^{k+1}$, so in order to conclude the proof we only
have to see that $G$ is strongly convex. Given $x, v\in\R^{k+1}$ with $v\neq 0$,
it is enough to check that the function
    $$
h(t):=G(x+tv)=\widetilde{M}_{\varepsilon}(\beta(t), \gamma(t)),
    $$
where $\beta(t)=g(x_1 +tv_1, ..., x_k+tv_k)$ and
$\gamma(t)=\alpha(x_{k+1} + tv_{k+1})$, satisfies $h''(t)>0$.
If $v_{k+1}\neq 0$ and $(v_{1}, ..., v_{k})\neq 0$ then,
since $g$ is strongly convex on $\R^k$ and $\alpha$ is strongly
convex on $\R$, we have $\beta''(t)>0$ and $\gamma''(t)>0$, so exactly as in the proof of $(9)$
of Proposition \ref{properties of M(f,g)} we also get
$h''(t)>0$. On the other hand, if for instance we have $v_{k+1}=0$ then $\beta''(t)>0$ and $\gamma'(t)=\gamma''(t)=0$, so
\begin{eqnarray*}
& &\frac{d^2}{dt^2}\widetilde{M}_{\varepsilon}(\beta(t), \gamma(t))=\\
& &\frac{ \left(1+\theta_{\varepsilon}'(\beta(t)-\gamma(t))\right) \beta''(t) + \theta_{\varepsilon}''(\beta(t)-\gamma(t)) \left( \beta(t)-\gamma(t)\right)^{2}}{2}>0,
\end{eqnarray*}
because $|\widetilde{\theta}_{\varepsilon}'|\leq 1$, $\widetilde{\theta}_{\varepsilon}''>0$, and $\widetilde{\theta}_{\varepsilon}'(0)=0$. Similarly one checks that $\frac{d^2}{dt^2}\widetilde{M}_{\varepsilon}(\beta(t), \gamma(t))>0$ in the case when $(v_1,...,v_k)=0\neq v_{k+1}$.
\end{proof}

\begin{lemma}\label{reduction to Rk with k less than n}
Let $U\subseteq\R^n$ be open and convex, $f:U\to\R$ be a $C^p$ convex function, and $x_0\in U$.
Assume that $f$ is not supported at $x_0$ by any
$(n+1)$-dimensional corner function. Then there exist $k<n$, a
linear projection $P:\R^n\to\R^k$, a $C^p$ convex function
$c:P(U)\subseteq\R^k\to\R$, and a linear function $\ell:\R^n\to\R$ such that
$f=c\circ P+\ell$.
\end{lemma}
\begin{proof}
If $f$ is affine the result is obvious. If $f$ is not affine then
there exists $y_0\in U$ with $f'(x_0)\neq f'(y_0)$. It is clear
that $L_1(x, x_{n+1})=x_{n+1}-f'(x_0)(x)$ and $L_2(x,
x_{n+1})=x_{n+1}-f'(y_0)(x)$ are two linearly independent linear
functions on $\R^{n+1}$, hence $f$ is supported at $x_0$ by the
two-dimensional corner $x\mapsto \max\{f(x_0)+f'(x_0)(x-x_0),
f(y_0)+f'(y_0)(x-y_0)\}$. Let us define $m$ as the greatest
integer number so that $f$ is supported at $x_0$ by an
$m$-dimensional corner. By assumption we have $2\leq m<n+1$.
Define $k=m-1$. There exist $\ell_1, ..., \ell_{k+1}\in
(\R^n)^{*}$ with $L_{j}(x, x_{n+1})=x_{n+1}-\ell_j(x)$, $j=1, ...,
k+1$, linearly independent in $(\R^{n+1})^{*}$, and $b_{1}, ...,
b_{k+1}\in\R$, so that $C=\max_{1\leq j\leq k+1}\{\ell_j +b_j\}$
supports $f$ at $x_0$.

Observe that the $\{L_{j}-L_1\}_{j=2}^{k+1}$ are linearly
independent in $(\R^{n+1})^{*}$, hence so are the
$\{\ell_{j}-\ell_1\}_{j=2}^{k+1}$ in $(\R^n)^{*}$, and therefore
$\bigcap_{j=2}^{k+1}\textrm{Ker}\, (\ell_{j}-\ell_1)$ has
dimension $n-k$. Then we can find linearly independent vectors
$w_1, ..., w_{n-k}$ such that $\bigcap_{j=2}^{k+1}\textrm{Ker}\,
(\ell_{j}-\ell_1)=\textrm{span}\{w_1, ..., w_{n-k}\}$.

Now, given any $y\in U$, if $\frac{d}{dt} (f-\ell_1)(y+t
w_q)|_{t=t_{0}}\neq 0$ for some $t_0$ then
$f'(y+t_0 w_q)-\ell_1$ is linearly independent with
$\{\ell_{j}-\ell_1\}_{j=2}^{k+1}$, which implies that $(x,
x_{n+1})\mapsto x_{n+1}-f'(y+t_0w_q)$ is linearly independent with
$L_1, ..., L_{k+1}$, and therefore the function
$$x\mapsto \max\{\ell_1(x)+ b_1, ..., \ell_{k+1}(x)+b_{k+1},
f'(y+t_0 w_q)(x-y-t_0 w_q)+f(y+t_0 w_q)\}$$ is a $(k+2)$-dimensional corner supporting $f$ at $x_0$, which
contradicts the choice of $m$. Therefore we must have
    $$
\frac{d}{dt}(f-\ell_1)(y+tw_q)=0 \,\,\,
\textrm{ for all } \, y\in U, t\in\R \, \textrm{ with } \, y+tw_q\in U,  \, q=1, ..., n-k.
    $$
This implies that
$$
(f-\ell_1)(y+\sum_{j=1}^{n-k}t_{j}w_j)=(f-\ell_1)(y)
$$ if $y\in U$ and $y+\sum_{j=1}^{n-k}t_{j}w_j\in U$.
Let $Q$ be the orthogonal projection of $\R^n$ onto the subspace
$E:=\textrm{span}\{w_1, ..., w_{n-k}\}^{\bot}$. For each $z\in Q(U)$ we may define
$$
\widetilde{c}(z)=(f-\ell_1)(z+\sum_{j=1}^{n-k}t_j w_j)
$$
if $z+\sum_{j=1}^{n-k}t_j w_j \in U$ for some $t_1, ..., t_{n-k}$. It is clear that $\widetilde{c}:Q(U)\to\R$ is well defined,
convex and $C^p$, and satisfies
$$
f-\ell_1=\widetilde{c}\circ Q.
$$
Then, by taking a linear
isomorphism $T:E\to\R^k$ and setting $P=TQ$, we have that
$f=c\circ P+\ell_1$, where $c=\widetilde{c}\circ T^{-1}$ is defined on $P(U)$.
\end{proof}

Now we can prove Theorem \ref{uniform approximation of convex by
real analytic convex}. We already know that a convex function
$f:U\subseteq\R\to\R$ can be uniformly approximated from below by
$C^1$ functions, so we may assume that $f\in C^{1}(U)$.
We will proceed by induction on $n$, the dimension of $\R^n$.

For $n=1$ the result can be proved as follows. Either
$f:U\to\R$ is affine (in which case we are done) or $f$
can be supported by a $2$-dimensional corner at every point $x\in
U$. In the latter case, let us consider a compact interval
$I\subset U$. Given $\varepsilon>0$, since $f$ is convex and
Lipschitz on $I$ we can find finitely many affine functions $h_1,
..., h_m:\R\to\R$ such that each $h_j$ supports $f-\varepsilon$ at
some point $x_j\in I$ and $f-2\varepsilon\leq \max\{h_1, ...,
h_m\}$ on $I$. By convexity we also have $\max\{h_1, ..., h_{m}\}\leq f-\varepsilon$
on all of $U$. For each $x_j$ we may find a $2$-dimensional corner
$C_j$ which supports $f-\varepsilon$ at $x_j$. Since $f$ is
differentiable and convex we have $h_j=C_j$ on a neighborhood of
$x_j$ and, by convexity, also $h_j\leq C_j\leq f-\varepsilon$ and
$\max\{C_1, ..., C_m\}\leq f-\varepsilon$ on $U$. And we also have
$f-2\varepsilon\leq \max\{h_1,..., h_m\}\leq\max\{C_1, ...,
C_{m}\}\leq f-\varepsilon$ on $I$. Now apply Lemma \ref{strongly
convex approximation of corners} to find $C^{\infty}$ strongly
convex functions $g_{1}, ..., g_{m}:\R\to\R$ such that $C_j\leq
g_j\leq C_j +\varepsilon'$, where $\varepsilon':=\varepsilon/2m$,
and define $g:\R\to\R$ by
$$g=M_{\varepsilon'}(g_1, M_{\varepsilon'}(g_2,
M_{\varepsilon'}(g_3, ..., M_{\varepsilon'}(g_{m-1}, g_m))...))$$
(for instance, if $m=3$ then $g=M_{\varepsilon'}(g_1,
M_{\varepsilon'}(g_2, g_3))$ ). By Proposition \ref{properties of
M(f,g)} we have that $g\in C^{\infty}(\R)$ is strongly convex,
    $$
\max\{C_1, ..., C_m\}\leq g\leq \max\{C_1, ...,
C_m\}+m\varepsilon'\leq f-\frac{\varepsilon}{2} \,\,\, \textrm{ on
} \,\,\, U,
    $$
and
    $$
f-2\varepsilon\leq \max\{C_1, ..., C_m\}\leq g \,\,\, \textrm{ on
} \,\,\, I.
    $$
Therefore $f:U\subseteq\R\to\R$ can be approximated from below by $C^\infty$
strongly convex functions, uniformly on compact subintervals of
$U$. By Theorem \ref{main theorem} and Remark \ref{the method
preserves strong convexity etc} we conclude that, given
$\varepsilon>0$ we may find a $C^\infty$ strongly convex function
$h$ such that $f-2\varepsilon\leq h\leq f-\varepsilon$ on $U$.

Finally, set $\eta(x)=\frac{1}{2}\min\{h''(x), \varepsilon\}$ for
every $x\in U$. The function $\eta:U\to (0, \infty)$ is
continuous, so we can apply Whitney's theorem on $C^2$-fine
approximation of $C^2$ functions by real analytic functions to find a real analytic function $g:U\to\R$ such
that
    $$
\max\{|h-g|, |h'-g'|, |h''-g''|\}\leq \eta.
    $$
This implies that $f-3\varepsilon\leq g\leq f$ and
$g''\geq\frac{1}{2}h''>0$, so $g$ is strongly convex as well.

\medskip

Now assume the result is true in $\R, \R^2, ..., \R^{d}$, and let
us see that then it is also true in $\R^{d+1}$. If there is some
$x_0\in U$ such that $f:U\subseteq\R^{d+1}\to\R$ is not supported
at $x_0$ by any $(d+2)$-dimensional corner function then,
according to Lemma \ref{reduction to Rk with k less than n}, we
can find $k\leq d$, a linear projection $P:\R^{d+1}\to\R^k$, a
linear function $\ell:\R^{d+1}\to\R$, and a $C^\infty$ convex
function $c:P(U)\to\R$ such that $f=c\circ P+\ell$. By assumption
there exists a real analytic convex function $h:P(U)\subseteq\R^k\to\R$ so that
$c-\varepsilon\leq h\leq c$. Then the function $g=h\circ P+\ell$
is real analytic, convex (though never strongly convex), and
satisfies $f-\varepsilon\leq g\leq f$.

If there is no such $x_0$ then one can repeat exactly the same
argument as in the case $n=1$, just replacing $2$-dimensional
corners with $(d+2)$-dimensional corners, the interval $I$ with a compact
convex body $K\subset U$, and $\eta$ with
    $$
\eta(x)=\frac{1}{2}\min\{ \varepsilon, \, \min\{D^{2}h(x)(v)^2 \,
: v\in \R^{d+1}, \|v\|=1\}\},
    $$
in order to conclude that there exists a real analytic strongly
convex $g:U\to\R$ such that $f-\varepsilon\leq g\leq f$ on $U$.
\,\,\, $\Box$

Incidentally, the above argument also shows Proposition
\ref{characterization of functions that cannot be approximated by
strongly convex functions} in the case when $f$ is $C^1$. In the
general case of a nonsmooth convex function one just needs to take
two more facts into account. First, Lemma \ref{reduction to Rk
with k less than n} holds for nonsmooth convex functions (to see
this, use the fact that if the range of the subdifferential of a
convex function is contained in $\{0\}$ then the function is
constant, see for instance \cite[Chapter 1, Corollary
2.7]{Clarke}, and apply this to the function $(t_1,...,
t_{n-k})\mapsto (f-\ell_1)(y+\sum_{j=1}^{n-k}t_j w_j)$). Second, in the
above proof one can use Rademacher's theorem and uniform
continuity of $f$ to see that the $x_j$ can be assumed to be
points of differentiability of $f$.

\section{Three counterexamples}

In this section we briefly discuss the possibility of approximating a convex function $f:\R^n\to\R$ by smooth convex functions in the $C^{0}$-fine topology. We will see that there is quite a big difference between the cases $n=1$ and $n\geq 2$.

In the case $n=1$ it can be shown that every convex function $f:\R\to\R$ can be approximated by convex real analytic functions in this topology.\footnote{We will provide a proof of this statement in a forthcoming paper.} That is, for every continuous function $\varepsilon:\R\to (0, \infty)$ there exists a real analytic convex function $g:\R\to\R$ such that $|f(x)-g(x)|\leq \varepsilon(x)$ for all $x\in\R$. However, this approximation cannot be performed from below:
\begin{example}
Let $f:\R\to\R$ be defined by $f(x)=|x|$. For every $C^1$ convex function $g:\R\to\R$ such that $g(0)\leq 0$ we have $$\liminf_{|x|\to\infty}|f(x)-g(x)|>0.$$
\end{example}
\noindent In particular, if $\varepsilon:\R\to (0, \infty)$ is continuous and satisfies $\lim_{|x|\to\infty}\varepsilon(x)=0$ then there is no $C^1$ convex function $g:\R\to\R$ such that $|x|-\varepsilon(x)\leq g(x)\leq |x|$.

In two or more dimensions the situation gets much worse: $C^{0}$-fine approximation of convex functions by $C^1$ convex functions is no longer possible in general.

\begin{example}
For $n\geq 2$, let $f:\R^{n}\to\R$ be defined by $f(x_1, ..., x_n)=|x_1|$, and let $\varepsilon:\R^n\to (0, \infty)$ be continuous with $\lim_{|x|\to\infty}\varepsilon(x)=0$. Then there is no $C^1$ convex function $g:\R^n\to\R$ such that $|f-g|\leq\varepsilon$.
\end{example}
\begin{proof}
Suppose that such a function $g$ exists. Since $|g(x_1,0, ..., 0)- |x_1||\leq \varepsilon(x_1,0, ..., 0)$ and $\lim_{|x_1|\to\infty}\varepsilon(x_1,0, ..., 0)=0$, we know from the preceding example that $g(0)>0$. Consider the function $h:\R^{n-1}\to\R$ defined by $h(y)=g(0,y_1, ..., y_{n-1})$. We have $\lim_{|y|\to\infty}h(y)=0$ and $h(0)>0$, but this contradicts the fact that $h$ is convex.
\end{proof}

Our last example shows that for $n\geq 2$ there also exist $C^{\infty}$ convex functions on $\R^n$ which cannot be approximated by real analytic convex functions in the $C^{0}$-fine topology.

\begin{example}
Let $n\geq 2$ and $f:\R^n\to\R$ be defined by $f(x_1,..., x_n)=\alpha(x_1)$, where $\alpha:\R\to [0, \infty)$ is a $C^{\infty}$ convex function such that $\alpha(t)=0$ for $|t|\leq 1$, and $\alpha(t)=|t|-2$ for $|t|\geq 3$. Let $\varepsilon:\R^n\to (0, \infty)$ be continuous and such that $\lim_{|x|\to\infty}\varepsilon(x)=0$. Then there is no real analytic convex function $g:\R^n\to\R$ such that $|f-g|\leq\varepsilon$.
\end{example}
\begin{proof}
Suppose there exists such a function $g$. As in the preceding example, if $g(t,0)>\alpha(t)$ for some $t$ we easily get a contradiction. Therefore $t\mapsto g(t,0)$ is everywhere below $\alpha$. By convexity, if $\alpha(t_0)-g(t_0,0):=s_0>0$ for some $t_0\geq 3$, then $\alpha(t)-g(t,0)\geq s_0>0$ for all $t\geq t_0$, and therefore $g$ cannot approximate $f$ as required. Hence $g(t,0)=t-2$ for all $t\geq 3$, and since $g$ is real analytic, also $g(t,0)=t-2$ for all $t\in\R$. But then $\lim_{t\to-\infty}|g(t,0)-f(t,0)|=\infty$, and again $g$ cannot approximate $f$.
\end{proof}

\section{Appendix: convex functions vs convex bodies}

In this appendix we recall a (somewhat unbalanced) basic relationship between convex functions and convex bodies, regarding approximation. Given a convex function $f:\R^n\to\R$, if we consider the epigraph $C$ of $f$, which is an unbounded convex body in $\R^{n+1}$, we can approximate $C$ by smooth convex bodies $D_k$ such that $\lim_{k\to\infty} D_k=C$ in the Hausdorff distance. Then it is easy to see (via the implicit function theorem) that the boundaries $\partial D_k$ are graphs of smooth convex functions $g_k:\R^n\to\R$ such that $\lim_{k\to\infty} g_k=f$ uniformly on compact subsets of $\R^n$. But when $f$ is not Lipschitz this convergence is not uniform on $\R^n$, as the following example shows.

\begin{example}
Consider the function $f:\R\to\R$, $f(x)=x^2$. The epigraph $C:=\{(x,y): y\geq x^2\}$ is an unbounded convex body, and the set $D:=\{(x,y) : \textrm{dist} \left( (x,y), C\right)\leq \varepsilon/2\}$ is a $C^1$ convex body such that $C\subset D\subset C+\varepsilon B$, where $B$ is the unit ball of $\R^2$. Hence $D$ approximates $C$ in the Hausdorff distance, and the boundary $\partial D$ is indeed the graph of a $C^1$ convex function $g:\R\to\R$. But the function $g$ does not approximate $f$ on $\R$, because $\lim_{|x|\to\infty} |f(x)-g(x)|=\infty$.
\end{example}

Therefore one cannot employ results on approximation of (unbounded) convex bodies to deduce results on global approximation of convex functions.
By contrast, one can use the well known results on global approximation of Lipschitz convex functions by real analytic convex functions to deduce
the following.
\begin{corollary}[Minkowski]
Let $C\subset\R^n$ be a (not necessarily bounded) convex body. For every $\varepsilon>0$ there exists a real analytic convex body $D$ such that
$$
C\subset D\subset C+\varepsilon B,
$$
where $B$ is the unit ball of $\R^n$.
\end{corollary}
\begin{proof}
Consider the $1$-Lipschitz, convex function $f:\R^n\to [0, \infty)$ defined by $f(x)=\textrm{dist}(x,C)$. Using integral convolution with the heat kernel one can produce a real analytic convex (and $1$-Lipschitz) function $g:\R^n\to\R$ such that $f-2\varepsilon/3\leq g\leq f-\varepsilon/3$ on $\R^n$. Define $D=g^{-1}(-\infty, 0]$. Since $g$ is convex and does not have any minimum on $\partial D=g^{-1}(0)$, we have $\nabla g(x)\neq 0$ for all $x\in \partial D$, hence $\partial D$ is a $1$-codimensional real analytic submanifold of $\R^n$. Because $f\geq g$, we have $C\subset D$. And if $x\notin C+\varepsilon B$ then $f(x)\geq \varepsilon$, hence $g(x)-\varepsilon/3\geq f(x)-\varepsilon\geq 0$, which implies $g(x)>0$, that is $x\notin D$.
\end{proof}

\bigskip


\end{document}